\documentclass[final,leqno]{siamltex}
\pdfoutput=1

\usepackage{amsmath,amssymb}
\usepackage{mathtools}
\usepackage{microtype}

\usepackage{graphicx}
\usepackage{grffile}
\usepackage[labelformat=empty,margin=1em,captionskip=0pt,farskip=0pt]{subfig}

\let\oldthebibliography\thebibliography
\usepackage[longnamesfirst,numbers]{natbib}
\shortcites{GaSaSk1999}
\let\thebibliography\oldthebibliography

\usepackage{hyperref}
\hypersetup{
  pdftitle={Modified trigonometric integrators},
  pdfauthor={Robert I. McLachlan and Ari Stern},
  pdfsubject={MSC 2010: 65P10, 70K70},
  pdfkeywords={trigonometric integrators, geometric integrators, highly
    oscillatory problems, IMEX, Fermi--Pasta--Ulam},
  bookmarksopen=false
}

\title{Modified trigonometric integrators}

\author{Robert~I. McLachlan%
  \thanks{IFS, Massey University, Palmerston North, New Zealand 5301
    (\texttt{r.mclachlan@massey.ac.nz})} \and Ari Stern%
  \thanks{Department of Mathematics, Washington University in
    St.~Louis, Campus Box 1146, One Brookings Drive, Saint Louis,
    Missouri 63130 (\texttt{astern@math.wustl.edu})}}

\begin{document}

\maketitle

\begin{abstract}
  We study modified trigonometric integrators, which generalize the
  popular class of trigonometric integrators for highly oscillatory
  Hamiltonian systems by allowing the fast frequencies to be
  modified. Among all methods of this class, we show that the IMEX
  (implicit-explicit) method, which is equivalent to applying the
  midpoint rule to the fast, linear part of the system and the
  leapfrog (St\"ormer/Verlet) method to the slow, nonlinear part, is
  distinguished by the following properties: (i) it is symplectic;
  (ii) it is free of artificial resonances; (iii) it is the unique
  method that correctly captures slow energy exchange to leading
  order; (iv) it conserves the total energy and a modified oscillatory
  energy up to to second order; (v) it is uniformly second-order
  accurate in the slow components; and (vi) it has the correct
  magnitude of deviations of the fast oscillatory energy, which is an
  adiabatic invariant. These theoretical results are supported by
  numerical experiments on the Fermi--Pasta--Ulam problem and indicate
  that the IMEX method, for these six properties, dominates the class
  of modified trigonometric integrators.
\end{abstract}



\section{Introduction}

\subsection{Overview} 
Over the past two decades, there has been considerable interest in
so-called \emph{geometric numerical integrators}, particularly
symplectic integrators for Hamiltonian systems
\citep{HaLuWa2006,LeRe2004}.  In contrast to general-purpose numerical
integrators, geometric integrators are designed especially to be
applied to systems with some additional underlying structure
(symmetries, invariants, etc.) that the algorithm must preserve
exactly, at least up to round-off error.

While there have been many successes in this area, the integration of
highly oscillatory Hamiltonian systems---which feature both stiff,
linear forces and soft, nonlinear forces---has remained persistently
difficult, due to the simultaneous presence of fast and slow time
scales.  Such systems are especially prevalent, for instance, in
molecular dynamics, where one must contend with strong, short-range
bonding forces, as well as weak, long-range electrostatic forces.

One of the main advances has been the development and analysis of
\emph{trigonometric integrators}, which are numerical methods designed
especially to integrate certain highly oscillatory systems.  However,
these methods have certain drawbacks: in particular, there is a
trade-off between numerical stability and consistency with respect to
certain dynamical features, such as the emergent multiscale phenomenon
of \emph{slow energy exchange} and the near-preservation of adiabatic
invariants.

In this paper, we show that \emph{modified trigonometric
  integrators}---that is, trigonometric integrators with modified
oscillatory frequencies---provide a way around this obstacle. Naively,
one might expect that perturbing the frequency would increase the
error, so it may seem counterintuitive to suggest that this can
actually improve numerical performance.  Yet, we show that this is
indeed the case: allowing for the frequency to be modified provides an
additional degree of freedom, which makes it possible to sidestep the
aforementioned trade-off between stability and multiscale structure
preservation. Specifically, we show that a particular, unique choice
of modified frequency yields an integrator that is both stable and
structure-preserving, and this is precisely the implicit-explicit
(IMEX) method of \citet{StGr2009}.

\subsection{The numerical challenge of fast oscillations}
\label{sec:fastOsc}

Consider a prototypical highly oscillatory problem, given by the
second-order equation
\begin{equation}
  \label{eqn:secondOrder}
  \ddot{q} + \Omega ^2 q = g ( q) ,
\end{equation} 
where $ q (t) \in \mathbb{R}^{d} $ is a trajectory, $ \Omega =
\bigl( \begin{smallmatrix} 0 & 0 \\ 0 & \omega I
\end{smallmatrix} \bigr) $ is a $ {d} \times {d} $ matrix with
constant fast frequency $ \omega \gg 1 $, and $ g \colon
\mathbb{R}^{d} \rightarrow \mathbb{R}^{d} $ is a conservative
nonlinear force, so that $ g = - \nabla U $ for some scalar potential
$ U \colon \mathbb{R}^{d} \rightarrow \mathbb{R} $.  The equation
\eqref{eqn:secondOrder} can also be written as a first-order system on
$ (q,p) \in \mathbb{R} ^{ 2 {d} } $,
\begin{equation} 
  \label{eqn:firstOrder}
  \begin{aligned}
    \dot{q} &= p ,\\
    \dot{p} &= - \Omega ^2 q + g (q) ,
  \end{aligned}
\end{equation} 
which are Hamilton's equations for the separable Hamiltonian
\begin{equation*}
  H ( q, p ) = \frac{1}{2} \lVert p \rVert ^2 + \frac{1}{2} \lVert
  \Omega q \rVert ^2 + U (q) .
\end{equation*}
Due to this underlying Hamiltonian structure, it is desirable to use a
symplectic integrator to obtain numerical solutions to
\eqref{eqn:secondOrder}--\eqref{eqn:firstOrder}.

One of the most popular, widely used symplectic integrators is the
\emph{St\"ormer/Verlet} (or \emph{leapfrog}) \emph{method}, which
discretizes \eqref{eqn:secondOrder} by the centered finite-difference
equation
\begin{equation}
  \label{eqn:stoermerVerlet}
  \frac{ q _{ n + 1 } - 2 q _n + q _{ n -1 } }{ h ^2 } + \Omega ^2 q
  _n = g ( q _n ) ,
\end{equation} 
where $h$ denotes the time step size.  An equivalent approximation for
the first-order system \eqref{eqn:firstOrder} is given by the
symmetric algorithm
\begin{align*}
  p _{ n + 1/2 } &= p _n + \frac{1}{2} h \bigl[ - \Omega ^2 q _n + g
  ( q _n ) \bigr] ,\\
  q _{ n + 1 } &= q _n + h p _{ n + 1/2 },\\
  p _{ n + 1 } &= p _{ n + 1/2 }  + \frac{1}{2} h \bigl[ - \Omega
  ^2 q _{n+1} + g ( q _{n+1} ) \bigr],
\end{align*} 
which is sometimes called the \emph{velocity Verlet method}.  Note
that, if \eqref{eqn:stoermerVerlet} is used to compute a numerical
trajectory $ ( \ldots, q _{ n -1 } , q _n , q _{ n + 1 } , \ldots ) $,
then we can still recover $ p _n $, after the fact, by taking $ p _n =
\frac{ q _{ n + 1 } - q _{ n -1 } }{ 2 h } $.  This method corresponds
to splitting the Hamiltonian into kinetic and potential components, $
H ( q, p ) = T (p) + V (q) $, where
\begin{equation*}
  T (p) = \frac{1}{2} \lVert p \rVert ^2, \qquad V (q) = \frac{1}{2} \lVert
  \Omega q \rVert ^2 + U (q),
\end{equation*}
and alternating between the purely kinetic flow of $ T (p) $ and the
purely potential flow of $ V (q) $.  This is an example of a
\emph{splitting method} (see~\citet{McQu2002}), and since the flows of
$T (p)$ and $V (q) $ are each Hamiltonian, the composition $ ( q _n ,
p _n ) \mapsto ( q _{ n + 1 }, p _{ n + 1 } ) $ is a symplectic map.

Despite these desirable geometric properties, however, the
St\"ormer/Verlet method cannot integrate highly oscillatory systems
efficiently.  As an explicit method, it remains stable only for time
steps on the order $ h = \mathcal{O} ( \omega ^{-1} ) $; in
particular, when $ g = 0 $, we have the linear stability condition $ h
\omega \leq 2 $.  Therefore, to integrate over a time interval of
fixed size, the method requires $ \mathcal{O} (\omega) $ time steps,
and hence $ \mathcal{O} ( \omega ) $ evaluations of the nonlinear
force $g$, which becomes prohibitively expensive for large $\omega$.

A typical implicit method encounters similar difficulties.  For
example, the \emph{implicit midpoint method} discretizes
\eqref{eqn:firstOrder} by the one-step algorithm
\begin{align*}
  q _{ n + 1 } &= q _n + h \Bigl( \frac{ p _n + p _{ n + 1 } }{ 2 }
  \Bigr) ,\\
  p _{ n + 1 } &= p _n + h \biggl[ - \Omega ^2 \Bigl( \frac{ q _n + q
    _{ n + 1 } }{ 2 } \Bigr) + g \Bigl( \frac{ q _n + q _{ n + 1 } }{
    2 } \Bigr) \biggr],
\end{align*}
which is equivalent to the centered finite-difference scheme
\begin{equation}
\label{eqn:midpoint}
\begin{multlined}[4in]
  \frac{ q _{ n + 1 } - 2 q _n + q _{ n -1 } }{ h ^2 } + \Omega ^2
  \Bigl( \frac{ q _{ n + 1 } + 2 q _n + q _{ n -1 } }{ 4 } \Bigr) \\
  = \frac{1}{2} g \Bigl( \frac{ q _n + q _{ n + 1 } }{ 2 } \Bigr) +
  \frac{1}{2} g \Bigl( \frac{ q _{ n -1 }+ q _n }{ 2 } \Bigr)
\end{multlined}
\end{equation}
for the second-order equation \eqref{eqn:secondOrder}.  While this
method is linearly unconditionally stable, it requires a nonlinear
solve at each time step, since $ q _{ n + 1 } $ appears inside the
nonlinear force $g$.  However, a numerical solver (e.g., Newton's
method) will require several evaluations of $g$ at each time step, so
this method is also computationally expensive.

The failure of these traditional symplectic integrators has motivated
the development of numerical methods designed especially for highly
oscillatory Hamiltonian systems.  The goal of this research has been
to obtain efficient, explicit integrators that are stable and accurate
for large time steps $h$.  By ``large time steps,'' we mean that the
step size can be chosen independently of the fast frequency $\omega$,
so that $ h ^{-1} = \mathcal{O} (1) $ as $ \omega \rightarrow \infty
$.  Therefore, in sharp contrast to the St\"ormer/Verlet method, such
integrators require only $ \mathcal{O} (1) $ evaluations of the
nonlinear force, rather than $ \mathcal{O} (\omega) $.

\subsection{Trigonometric integrators} Trigonometric integrators are
designed to integrate \eqref{eqn:secondOrder}--\eqref{eqn:firstOrder}
exactly when the nonlinear force $g$ vanishes, i.e., when the system
reduces to a harmonic oscillator.  Let $\psi$ and $\phi$ be a pair of
even, real-valued \emph{filter functions} satisfying $ \psi (0) = \phi
(0) = 1 $, and denote $ \Psi = \psi ( h \Omega ) $, $ \Phi = \phi ( h
\Omega ) $, and $ g _n = g ( \Phi q _n ) $.  Then the
\emph{trigonometric integrator} corresponding to the filters $ \psi $,
$\phi$, is defined by the difference equation
\begin{equation}
  \label{eqn:trigIntegrator}
  q _{ n + 1 } - 2 \cos ( h \Omega ) q _n + q _{ n -1 } = h ^2 \Psi g _n .
\end{equation} 
We can extend this to a symmetric, one-step method for $ (q,p) \in
\mathbb{R} ^{ 2 {d} } $ by introducing a new filter function $ \psi _1
$, which satisfies $ \psi (\xi) = \operatorname{sinc} (\xi) \psi _1
(\xi) $, and taking $ \Psi _1 = \psi _1 ( h \Omega ) $.  We then
obtain a velocity Verlet-like algorithm,
\begin{align*}
  p _n ^+ &= p _n + \frac{1}{2} h \Psi _1 g _n ,\\
  \begin{pmatrix}
    q _{ n + 1 } \\
    p _{ n + 1 } ^-
  \end{pmatrix} &=
  \begin{pmatrix}
    \cos (h \Omega) & h \operatorname{sinc} (h \Omega) \\
    - \Omega \sin (h \Omega) & \cos (h \Omega)
  \end{pmatrix}
  \begin{pmatrix}
    q _n \\
    p _n ^+
  \end{pmatrix},\\
  p _{ n + 1 } &= p _{ n + 1 } ^- + \frac{1}{2} h \Psi _1 g _{ n + 1 } .
\end{align*} 
Similarly to St\"ormer/Verlet, if \eqref{eqn:trigIntegrator} is used
to compute a numerical trajectory $ ( \ldots, q _{ n -1 } , q _n , q
_{ n + 1 } , \ldots ) $, then $ p _n $ can be recovered by taking $
\operatorname{sinc} ( h \Omega ) p _n = \frac{ q _{ n + 1 } - q _{ n
    -1 } }{ 2 h } $, as long as $ \operatorname{sinc} ( h \omega )
\neq 0 $. Whatever the choice of $ \psi $ and $\phi$, these methods
reduce to St\"ormer/Verlet in the case $ \omega = 0 $, and to the
exact solution of the harmonic oscillator in the case $ g = 0 $.

One of the simplest trigonometric integrators is the
\emph{Deuflhard/impulse method}\footnote{Christian Lubich pointed out
  to us that, although the Deuflhard and impulse methods are distinct
  in general, they happen to coincide in the case of highly
  oscillatory problems of this type.  This is the reason behind the
  slash in the name ``Deuflhard/impulse.''}, which corresponds to the
choice of filters $ \psi (\xi) = \operatorname{sinc} (\xi) $ (i.e., $
\psi _1 = 1 $) and $ \phi = 1 $.  In this case, the trigonometric
integrator corresponds to a splitting method: the Hamiltonian is split
as $ H (q,p) = H_{\text{fast}} (q,p) + U (q) $, where $
H_{\text{fast}} (q,p) = \frac{1}{2} \bigl( \lVert p \rVert ^2 + \lVert
\Omega q \rVert ^2 \bigr) $.  While the Deuflhard/impulse method has
many desirable properties, it has one major drawback: spurious
numerical resonances arise when $ h \omega $ is close to a nonzero
integer multiple of $\pi$, causing a loss of stability and accuracy.
This resonance instability causes serious problems whenever $ h \omega
\geq \pi $, so effectively, the method forces us to choose $ h =
\mathcal{O} ( \omega ^{-1} ) $, just like the St\"ormer/Verlet method,
making it unsuitable for integration with long time steps. (Similar
resonance phenomena also plague other impulse-type methods, including
multiple-time-stepping methods, cf.~\citet{BiSk1993}.)

In contrast to the ``unfiltered'' Deuflhard/impulse method, other
trigonometric integrators use $\phi$ to filter (or ``mollify'') the
slow force, so as to lessen the problem of resonance
instability. \emph{Mollified impulse methods} allow the filter $\phi$
to be chosen arbitrarily, and then take $ \psi (\xi) =
\operatorname{sinc} (\xi) \phi (\xi) $, i.e., $ \psi _1 = \phi $.
Like the Deuflhard/impulse method (which corresponds to the special
case $\phi = 1 $), these are also Hamiltonian splitting methods, where
the potential $U$ is replaced by the mollified potential $ \overline{
  U } (q) = U ( \Phi q ) $. Consequently, mollified impulse methods
are symplectic; in fact, it is straightforward to show that a
trigonometric integrator is symplectic if and only if $ \psi (\xi) =
\operatorname{sinc} (\xi) \phi (\xi) $, i.e., it is a mollified
impulse method.

Various trigonometric integrators, corresponding to different choices
of the filter functions $\psi$ and $\phi$, have appeared in the
literature, and are summarized in \autoref{tab:trigMethods}.  The
alphabetical labels for these methods (A--E and G) follow the
convention of \citet{HaLuWa2006}, which has since been adopted by
several other authors.  (We have omitted ``method F,'' from
\citep{HaLuWa2006}, which is a two-force method rather than a
trigonometric integrator.)  Of these, note that only method B
(Deuflhard/impulse) and method C (mollified impulse) satisfy the
symplecticity condition $ \psi (\xi) = \operatorname{sinc} (\xi) \phi
(\xi) $.

\begin{table}
  \begin{tabular}{clll}
    Method & $ \psi (\xi) $ & $ \phi (\xi) $ & Reference \\
    \hline\rule{0ex}{2.5ex}%
    A & $ \operatorname{sinc} ^2 (\frac{1}{2} \xi)
    $ & $  1 $ & \citet{Gautschi1961}\\
    B & $ \operatorname{sinc} (\xi) $ & $ 1 $ &
    \citet{Deuflhard1979}\\
    C & $ \operatorname{sinc} ^2 (\xi) $ & $ \operatorname{sinc} (\xi)
    $ & \citet{GaSaSk1999}\\
    {D} & $ \operatorname{sinc} ^2 (\frac{1}{2} \xi) $ &
    $ \operatorname{sinc} (\xi) \bigl(1 + \frac{ 1 }{ 3 } \sin ^2
    (\frac{1}{2} \xi) \bigr) $ & \citet{HoLu1999} \\
    E & $ \operatorname{sinc} ^2 (\xi) $ & $ 
    1 $ & \citet{HaLu2000} \\
    G & $ \operatorname{sinc} ^3 (\xi) $ & $ 
    \operatorname{sinc} (\xi) $ & \citet{GrHo2006}    
  \end{tabular}
  \vskip 1ex 
  \caption{Filter functions corresponding to various 
    trigonometric integrators.\label{tab:trigMethods}}
\end{table}

\subsection{Modulated Fourier expansion and slow exchange}
\label{sec:slowExchange}

The \emph{modulated Fourier expansion} is a powerful technique for
analyzing the dynamics of highly oscillatory systems, as well as the
numerical behavior of trigonometric integrators for such systems. We
give only a brief summary here; for a detailed treatment, see
\citet{HaLu2000,HaLuWa2006}.

Suppose that $ q (t) $ is a solution of the highly oscillatory system
\eqref{eqn:secondOrder}. To separate out its fast- and slow-scale
features, we approximate $ q (t) $ asymptotically by a trajectory $ x
(t) $ of the form
\begin{equation}
  \label{eqn:principalMFEexact}
  x (t) = y (t) + e ^{ i \omega t } z (t) + e ^{ - i \omega t }
  \overline{ z } (t) ,
\end{equation} 
where $ y (t) $ is real-valued and $ z (t) $ is complex-valued.
Assuming that the energy of $ x (t) $ is bounded on the time interval
of interest, this implies that $ z (t) = \mathcal{O} ( \omega ^{-1} )
$.  Next, we can decompose $ x = ( x _0, x _1 ) $, $ y = ( y _0, y _1
) $, and $ z = ( z _0, z _1 ) $, according to the blocks of $\Omega$.
Plugging $ x (t) $ into \eqref{eqn:secondOrder}, Taylor expanding $ g
(x) $ around $y$, and matching the terms on both sides up to $
\mathcal{O} ( \omega ^{ -3 } ) $ yields the system of equations
\begin{equation}
\label{eqn:coefficients}
\begin{aligned}
  \ddot{y} _0 &= g _0 \bigl( y _0 , \omega ^{ - 2 } g _1 ( y _0, 0 )
  \bigr) + \frac{ \partial ^2 g _0 }{ \partial x _1 ^2 } ( y _0, 0 ) (
  z _1 , \overline{z} _1 ) ,\\
  2 i \omega \dot{z} _1 &= \frac{ \partial g _1 }{ \partial x _1 } ( y
  _0, 0 ) z _1 .
\end{aligned}
\end{equation}
(The $ y _1 $ and $ z _0 $ components can both be eliminated up to
this order of accuracy.)  Here, $ y _0 $ evolves on the time scale $
\mathcal{O} (1) $ and describes the non-stiff dynamics of the system,
while $ z _1 $ evolves on the time scale $ \mathcal{O} ( \omega ) $
and corresponds to a multiscale phenomenon known as \emph{slow energy
  exchange}.  If $ I _j = \frac{1}{2} p _{ 1, j } ^2 + \frac{1}{2}
\omega ^2 q _{ 1, j } ^2 $ is the energy in the $j$th stiff component
of the system, then it can be shown that, up to $ \mathcal{O} ( \omega
^{-1} ) $, we have $ I _j \approx 2 \omega ^2 \lvert z _{ 1, j }
\rvert ^2 $.  Here, we have split $ q = ( q _0 , q _1 ) $ and $ p = (
p _0 , p _1 ) $ into non-stiff and stiff blocks, as above, and $ q _{
  1, j } $, $ p _{ 1, j } $, $ z _{ 1, j } $ denote the $j$th
components of the corresponding vectors $ q _1 $, $ p _1 $, $ z _1 $.
It follows that the evolution of $ z _1 $ describes the slow exchange
of energy between the stiff components, coupled through the nonlinear
force.  Moreover, the total stiff energy $ I = \sum _j I _j \approx 2
\omega ^2 \lVert z _1 \rVert ^2 $ is an adiabatic invariant, since
\begin{equation*}
  \frac{\mathrm{d}}{\mathrm{d}t} 2 \omega ^2 \lVert z _1
  \rVert ^2 = 4 \omega ^2 \operatorname{Re} \langle z _1,
  \dot{z} _1 \rangle = \mathcal{O}  ( \omega ^{-1} ) ,
\end{equation*} 
and therefore $ \dot{ I } = \mathcal{O} ( \omega ^{-1} ) $. Hence,
deviations in $I$ are also $ \mathcal{O} ( \omega ^{-1} ) $ over a
fixed time interval.

A similar technique can be applied to analyze numerical behavior. For
a trigonometric integrator with time step size $h$, the numerical
trajectory $ q _n $ can be approximated asymptotically by $ x _h ( n h
) $, where
\begin{equation}
  \label{eqn:principalMFEnumerical}
  x _h (t) = y _h (t) + e ^{ i \omega t } z _h (t) + e ^{ -i \omega
    t } \overline{ z } _h (t) .
\end{equation}
As we did for the continuous dynamics, we can plug this ansatz into
\eqref{eqn:trigIntegrator} and match terms, obtaining a system of
equations,
\begin{equation}
\label{eqn:trigCoefficients}
\begin{aligned}
  \delta _h ^2 y _{ h, 0 } &= g _0 \bigl( y _{ h,0 } , \gamma \omega
  ^{ - 2 } g _1 ( y _{ h, 0 } , 0 ) \bigr) + \beta \frac{ \partial ^2
    g _0 }{ \partial x _1 ^2 } ( y _{ h,0 }, 0 ) ( z _{ h,1 },
  \overline{z} _{ h, 1 } ) ,\\
  2 i \omega \dot{z} _{ h, 1 } &= \alpha \frac{ \partial g _1
  }{ \partial x _1 } ( y _{ h, 0 } , 0 ) z _{ h, 1 } ,
\end{aligned}
\end{equation}
which hold up to $ \mathcal{O} ( \omega ^{ -3} ) $.  Here, $ \delta _h
^2 $ denotes the second finite-difference operator, defined by
\begin{equation*}
  \delta _h ^2 y _{h,0}  (t) = \frac{ y _{h,0} (t+h) - 2 y _{h,0} (t)
    + y _{h,0} ( t - h )}{ h ^2 } ,
\end{equation*} 
while the constants $\alpha$, $\beta$, and $\gamma$ are given by
\begin{equation*}
  \alpha = \frac{ \psi ( h \omega ) \phi ( h \omega ) }{
    \operatorname{sinc} (h \omega ) } , \qquad \beta = \phi (h \omega
  ) ^2 , \qquad \gamma = \frac{ \psi ( h \omega ) \phi ( h \omega ) }{
    \operatorname{sinc} ^2 ( \frac{1}{2} h \omega ) } .
\end{equation*} 
Comparing \eqref{eqn:coefficients} and \eqref{eqn:trigCoefficients},
it follows that the dynamics of $ z _{ h , 1 } $ are consistent with
those for $ z _1 $ only if $ \alpha = 1 $.  Moreover, to fully capture
the coupled dynamics between $ y _0 $ and $ z _1 $, one would also
require $ \beta = 1 $ and $ \gamma = 1 $.

Of the methods listed in \autoref{tab:trigMethods}, only Method B, the
Deuflhard/impulse method, satisfies $ \alpha = 1 $.  However, as
discussed previously, the resonance instability of this method makes
it practically impossible to take large time steps.  Hence, we cannot
hope to model the slow-energy exchange accurately, using a
trigonometric integrator, unless we sacrifice either stability or
efficiency.  A more fundamental problem is that $ \alpha \neq \gamma $
in general, so even if we are willing to make the aforementioned
trade-off, it is impossible for a trigonometric integrator to satisfy
$ \alpha = \beta = \gamma = 1 $.

Multi-force methods provide one way around this obstacle, but as their
name suggests, they require multiple evaluations of the slow force $g$
per time step.  However, we show that there is another way around this
obstacle: by modifying the fast frequency $\omega$, it is possible for
a stable, efficient method to achieve $ \alpha = \beta = \gamma = 1 $,
with only a single evaluation of $g$ per time step.

\subsubsection*{Remark} Strictly speaking, \eqref{eqn:principalMFEexact}
and \eqref{eqn:principalMFEnumerical} contain only the principal
(i.e., leading-order) terms of the modulated Fourier expansion. While
the constant and $ e ^{ \pm i \omega t } $ terms are sufficient to
describe slow energy exchange, other properties---including long-time
conservation of the total and oscillatory energies---require further
expansion in the higher-order terms $ e ^{ \pm 2 i \omega t } $, $ e
^{ \pm 3 i \omega t } $, etc. Again, we refer the reader to
\citet{HaLu2000,HaLuWa2006} for a full account.

\subsection{Overview of results}
We begin, in \autoref{sec:modTrig}, by defining modified trigonometric
integrators, and by giving a few examples.  We then show how the
modulated Fourier expansion can be applied to these methods, and we
use this to derive consistency conditions for slow energy exchange.
The main result of this section, \autoref{thm:IMEX}, shows that
modified trigonometric integrators can indeed satisfy the full
consistency condition $ \alpha = \beta = \gamma = 1 $; in fact, we
prove that there is a unique modified trigonometric integrator which
does so, coinciding with the implicit-explicit (IMEX) method of
\citet{StGr2009}. Furthermore, \autoref{thm:energy} shows that IMEX
conserves total energy and a modified oscillatory energy up to $
\mathcal{O} ( h ^2 ) $.

\autoref{sec:experiments} presents the results of several numerical
experiments for the widely-studied and dynamically rich
Fermi--Pasta--Ulam test problem.  We compare the numerical behavior of
the trigonometric integrators listed in \autoref{tab:trigMethods} with
that of the IMEX modified trigonometric integrator.  These experiments
demonstrate the trade-off between stability, consistency, and accuracy
inherent to standard trigonometric integrators. By contrast, the IMEX
modified method performs well in all of these experiments, without any
observed trade-off, as predicted by \autoref{thm:IMEX} and
\autoref{thm:energy}.

Finally, one of these numerical experiments reveals that, although the
total oscillatory energy $ I = \sum _j I _j $ is well-conserved by all
of the methods considered (both modified and unmodified), the
integrators vary considerably with respect to the magnitude of
deviations in this adiabatic invariant.  In \autoref{sec:deviations},
we analyze the deviation in total oscillatory energy by examining
higher-order terms in the modulated Fourier expansion.  This analysis
provides a theoretical explanation for the behavior observed in the
numerical experiments.

\paragraph{Relationship to previous work} As mentioned above, the IMEX
method was introduced for highly oscillatory problems in
\citet{StGr2009}. This earlier paper focused primarily on the
variational, symplectic, and stability properties of IMEX, and on its
comparison with multiple-time-stepping methods (as opposed to
trigonometric integrators). It was observed that IMEX can be viewed as
a Deuflhard/impulse method with modified frequency, implying the
partial consistency condition $ \alpha = 1 $ for slow energy exchange
\citep[Theorem 4.1]{StGr2009}. \autoref{thm:IMEX} is a substantial
strengthening of this consistency result; the other results and
numerical experiments presented in the current paper are independent
of those in \citep{StGr2009}.

\section{Modified trigonometric integrators}
\label{sec:modTrig}

\subsection{Basic definitions} A \emph{modified trigonometric
  integrator} for the highly oscillatory system
\eqref{eqn:secondOrder} is defined by the second-order difference
equation
\begin{equation}
\label{eqn:modTrig}
  q _{ n + 1 } - 2 \cos ( h \widetilde{ \Omega } ) q _n + q _{ n -1 }
  = h ^2 \Psi  g _n ,
\end{equation} 
where $ \widetilde{ \Omega } = \bigl( \begin{smallmatrix} 0 & 0 \\ 0 &
  \widetilde{ \omega } I
\end{smallmatrix} \bigr) $ and $ \widetilde{ \omega } $ is called the
\emph{modified frequency}.  If $\psi$ and $\phi$ are even, real-valued
filter functions satisfying $ \psi (0) = \phi (0) = 1 $, we now take $
\Psi = \psi ( h \widetilde{ \Omega } ) $ and $ \Phi = \phi ( h
\widetilde{ \Omega } ) $, while as before, we denote $ g _n = g ( \Phi
q _n ) $.

Although $ \Omega $ is generally singular, we commit a slight abuse of
notation by taking $ \Omega ^{-1} \widetilde{ \Omega } $ to mean the
matrix $ \bigl( \begin{smallmatrix}
  I & 0 \\
  0 & (\widetilde{ \omega } / \omega ) I
\end{smallmatrix} \bigr) $.  Letting $ \Psi = \Omega ^{-1} \widetilde{
  \Omega } \operatorname{sinc} ( h \widetilde{ \Omega } ) \Psi _1 $,
we consider the following symmetric, one-step algorithm:
\begin{align*}
  p _n ^+ &= p _n + \frac{1}{2} h \Psi _1 g _n ,\\
  \begin{pmatrix}
    q _{ n + 1 } \\
    p _{ n + 1 } ^-
  \end{pmatrix} &=
  \begin{pmatrix}
    \cos (h \widetilde{ \Omega}) & h \Omega ^{-1} \widetilde{ \Omega }
    \operatorname{sinc} (h
    \widetilde{ \Omega}) \\
    - \Omega \sin (h \widetilde{ \Omega}) & \cos (h \widetilde{
      \Omega})
  \end{pmatrix}
  \begin{pmatrix}
    q _n \\
    p _n ^+
  \end{pmatrix},\\
  p _{ n + 1 } &= p _{ n + 1 } ^- + \frac{1}{2} h \Psi _1 g _{ n + 1 } .
\end{align*} 
As with standard trigonometric integrators, this method is symplectic
when the filters satisfy $ \psi _1 = \phi $; since this gives $ \Psi
_1 g _n = - \nabla \overline{ U } ( q _n ) $, by the chain rule, and
hence the integrator corresponds to a splitting method for the
modified Hamiltonian.  This symplecticity condition can also be
written as $ \psi ( h \widetilde{ \omega } ) = \omega
  ^{-1} \widetilde{ \omega } \operatorname{sinc} ( h \widetilde{
    \omega } ) \phi ( h \widetilde{ \omega } ) $.

If \eqref{eqn:modTrig} is used to compute a numerical trajectory $ (
\ldots, q _{ n -1 } , q _n , q _{ n + 1 }, \ldots ) $, then the $ p _n
$ can be recovered by taking $ \Omega ^{-1} \widetilde{ \Omega }
\operatorname{sinc} ( h \widetilde{ \Omega } ) p _n = \frac{ q _{ n +
    1 } - q _{ n -1 } }{ 2 h } $, as long as $ \operatorname{sinc} ( h
\widetilde{ \omega } ) \neq 0 $; note that this is slightly different
from the previous expression for a standard trigonometric integrator.
If we define the modified momentum $ \widetilde{ p } _n = \Omega ^{-1}
\widetilde{ \Omega } p _n $, then it follows that $
\operatorname{sinc} ( h \widetilde{ \Omega } ) \widetilde{ p } _n =
\Omega ^{-1} \widetilde{ \Omega } \operatorname{sinc} ( h \widetilde{
  \Omega } ) p _n = \frac{ q _{ n + 1 } - q _{ n -1 } }{ 2 h } $.
Hence, the numerical algorithm in $ ( q, \widetilde{ p } ) $
corresponds to
\begin{align*}
  \widetilde{ p } _n ^+ &= \widetilde{ p } _n + \frac{1}{2} h
  \widetilde{ \Psi } _1 g _n ,\\
  \begin{pmatrix}
    q _{ n + 1 } \\
    \widetilde{ p } _{ n + 1 } ^-
  \end{pmatrix} &=
  \begin{pmatrix}
    \cos (h \widetilde{ \Omega}) & h \operatorname{sinc} (h
    \widetilde{ \Omega}) \\
    - \widetilde{ \Omega } \sin (h \widetilde{ \Omega}) & \cos (h
    \widetilde{ \Omega})
  \end{pmatrix}
  \begin{pmatrix}
    q _n \\
    \widetilde{ p } _n ^+
  \end{pmatrix},\\
  \widetilde{ p } _{ n + 1 } &= \widetilde{ p } _{ n + 1 } ^- +
  \frac{1}{2} h \widetilde{ \Psi } _1 g _{ n + 1 } ,
\end{align*} 
where $ \widetilde{ \Psi } _1 = \Omega ^{-1} \widetilde{ \Omega } \Psi
_1 $, which implies $ \Psi = \operatorname{sinc} ( h \widetilde{
  \Omega } ) \widetilde{ \Psi } _1 $.  This is precisely a standard
trigonometric integrator for the modified frequency $ \widetilde{
  \omega } $.

\subsection{Examples} The first example of a modified trigonometric
integrator is simply a standard trigonometric integrator, where we
make the trivial choice of modified frequency $ \widetilde{ \omega } =
\omega $.

A more interesting, nontrivial example is the St\"ormer/Verlet
method. Observe that the finite-difference scheme
\eqref{eqn:stoermerVerlet} can be rewritten as
\begin{equation*}
  q _{ n + 1 } - 2 ( I - \tfrac{1}{2} h ^2 \Omega ^2 ) q _n + q _{ n -1 } = h
  ^2 g ( q _n ) .
\end{equation*}
If we choose $ \widetilde{ \omega } $ such that $ \sin ( \frac{1}{2} h
\widetilde{ \omega } ) = \frac{1}{2} h \omega $, then it follows that
$ 1 - \tfrac{1}{2} h ^2 \Omega ^2 = 1 - 2 \sin ^2 ( \tfrac{1}{2} h
\widetilde{ \Omega } ) = \cos ( h \widetilde{ \Omega } ) $, and therefore
\begin{equation*}
  q _{ n + 1 } - 2 \cos ( h \widetilde{ \Omega } )  q _n + q _{ n -1 } = h
  ^2 g ( q _n ) .
\end{equation*}
Hence, the St\"ormer/Verlet method is a modified trigonometric
integrator with the above choice of $ \widetilde{ \omega } $, and with
the filters $ \psi = \phi = 1 $. It should be observed, though, that $
\sin ( \frac{1}{2} h \widetilde{ \omega } ) = \frac{1}{2} h \omega $
has no solution when $ \frac{1}{2} h \omega > 1 $; we are again
limited by the linear stability condition $ h \omega \leq 2 $, as in
\autoref{sec:fastOsc}.

Note that, although \eqref{eqn:modTrig} coincides with
St\"ormer/Verlet, the approximation of $ p _n $ is different from that
used in velocity Verlet.  From $ \sin ( \frac{1}{2} h \widetilde{
  \omega } ) = \frac{1}{2} h \omega $, we obtain
\begin{equation*}
  \frac{ \widetilde{ \omega } }{ \omega } \operatorname{sinc} ( h
  \widetilde{ \omega } ) = \frac{ h \widetilde{ \omega }
    \operatorname{sinc} ( h \widetilde{ \omega } ) }{ h \omega } =
  \frac{ \sin ( h \widetilde{ \omega } ) }{ h \omega } = \frac{ 2 \sin
    (\tfrac{1}{2} h \widetilde{ \omega } ) \cos (\tfrac{1}{2} h
    \widetilde{ \omega } ) }{ h \omega } = \cos     (\tfrac{1}{2} h
  \widetilde{ \omega } ).
\end{equation*} 
Substituting this into $ \Omega ^{-1} \widetilde{ \Omega }
\operatorname{sinc} ( h \widetilde{ \Omega } ) p _n = \frac{ q _{ n +
    1 } - q _{ n -1 } }{ 2 h } $, it follows that the momentum $ p _n
$ satisfies $ \cos ( \frac{1}{2} h \widetilde{ \Omega } ) p _n =
\frac{ q _{ n + 1 } - q _{ n -1 } }{ 2 h } $.

A third, particularly interesting example is the
\emph{implicit-explicit} (or \emph{IMEX}) integrator first suggested
by \citet{ZhSk1997} as a ``cheap'' version of the implicit midpoint
method, and more recently introduced and analyzed by \citet{StGr2009}
as an attractive method in its own right for highly oscillatory
problems.  (An essentially similar method has also been applied to the
linear Schr\"odinger equation, cf.~\citet{DeFa2009}.)  Combining the
left-hand side of the implicit midpoint method \eqref{eqn:midpoint}
with the right-hand side of the St\"ormer/Verlet method
\eqref{eqn:stoermerVerlet}, and multiplying both sides by $ h ^2 $, we
get the IMEX method,
\begin{equation*}
  ( q _{ n + 1 } - 2 q _n + q _{ n -1 } ) + (\tfrac{1}{2} h \Omega) ^2
  ( q _{ n + 1 } + 2 q _n + q _{ n -1 } ) = h ^2 g ( q _n ) ,
\end{equation*} 
which is only linearly implicit, and hence avoids the difficulty of
solving a nonlinear equation for $ q _{ n + 1 } $.  Now, if we choose
$ \widetilde{ \omega } $ such that $ \tan ( \frac{1}{2} h \widetilde{
  \omega } ) = \frac{1}{2} h \omega $, then this becomes
\begin{equation*}
  ( q _{ n + 1 } - 2 q _n + q _{ n -1 } ) + \tan  ^2 ( \tfrac{1}{2} h
  \widetilde{ \Omega  } )  ( q _{ n + 1 } + 2 q _n + q _{ n -1 } ) = h
  ^2 g ( q _n ) ,
\end{equation*} 
or
\begin{equation*}
  q _{ n + 1 } - 2 \cos ( h \widetilde{ \Omega } ) q _n + q _{ n -1 }
  = h ^2   \cos ^2 ( \tfrac{1}{2} h \widetilde{ \Omega } ) g ( q _n ) .
\end{equation*} 
Hence, the IMEX method can be reframed as a modified trigonometric
method, with the modified frequency $ \widetilde{ \omega } $
satisfying $ \tan ( \frac{1}{2} h \widetilde{ \omega } ) = \frac{1}{2}
h \omega $, and with the filters $ \psi (\xi) = \cos ^2 ( \frac{1}{2}
\xi ) $ and $ \phi = 1 $. In contrast to St\"ormer/Verlet, it is
always possible to solve for this $ \widetilde{ \omega } $, since $
\arctan $ (unlike $ \arcsin $) is defined on the entire real
line. Note that the IMEX method is also symplectic, since it satisfies
the aforementioned condition
\begin{equation*}
  \omega ^{-1} \widetilde{ \omega } \operatorname{sinc} ( h
  \widetilde{ \omega } ) \phi ( h \widetilde{ \omega } ) = \frac{
    \tfrac{1}{2} h \widetilde{ \omega } }{ \tfrac{1}{2} h \omega }
  \operatorname{sinc} ( h \widetilde{ \omega } ) 
  = \frac{ \tfrac{1}{2} \sin ( h \widetilde{ \omega } ) }{ \tan
    (\tfrac{1}{2} h \widetilde{ \omega } ) } = \cos ^2 ( \tfrac{1}{2}
  h \widetilde{ \omega } ) = \psi ( h \widetilde{ \omega } ).
\end{equation*} 
See \citet{StGr2009} for a discussion of IMEX as a splitting method,
which also implies its symplecticity.

\subsection{Modulated Fourier expansion and slow exchange} We have
seen that a modified trigonometric integrator has the same form as a
standard trigonometric integrator with frequency $ \widetilde{ \omega
} $ (modulo the choice of $ \widetilde{ p } _n = \Omega ^{-1}
\widetilde{ \Omega } p _n $ instead of $ p _n $).  Therefore, applying
the modulated Fourier expansion, we get
\begin{equation}
  \label{eqn:modTrigCoefficients}
\begin{aligned}
  \delta _h ^2 y _{ h, 0 } &= g _0 \bigl( y _{ h,0 } , \widetilde{
    \gamma } \widetilde{ \omega } ^{ - 2 } g _1 ( y _{ h, 0 } , 0 )
  \bigr) + \widetilde{ \beta } \frac{ \partial ^2 g _0 }{ \partial x
    _1 ^2 } ( y _{ h,0 }, 0 ) ( z _{ h,1 },
  \overline{z} _{ h, 1 } ) ,\\
  2 i \widetilde{ \omega } \dot{z} _{ h, 1 } &= \widetilde{ \alpha }
  \frac{ \partial g _1 }{ \partial x _1 } ( y _{ h, 0 } , 0 ) z _{ h,
    1 } ,
\end{aligned}
\end{equation}
where
\begin{equation*}
  \widetilde{ \alpha } = \frac{ \psi ( h \widetilde{ \omega } ) \phi (
    h \widetilde{ \omega } ) }{ \operatorname{sinc} (h \widetilde{
      \omega } ) } , \qquad \widetilde{ \beta } = \phi (h \widetilde{ \omega }
  ) ^2 , \qquad \widetilde{ \gamma } = \frac{ \psi ( h \widetilde{
      \omega } ) \phi (
    h \widetilde{ \omega } ) }{ \operatorname{sinc} ^2 ( \frac{1}{2} h
    \widetilde{ \omega } ) } .
\end{equation*} 
Now, if we define $ \widetilde{ I } _j = \frac{1}{2} \lvert
\widetilde{ p } _{ 1, j } \rvert ^2 + \frac{1}{2} \widetilde{ \omega }
^2 \lvert q _{ 1, j } \rvert ^2 $, then we have previously seen that $
\widetilde{ I } _j \approx 2 \widetilde{ \omega } ^2 \lvert z _{1,j}
\rvert ^2 $.  However, we are interested not in the behavior of $
\widetilde{ I } _j $, but in that of the original stiff energies $ I
_j $.  Since $ \widetilde{ p } _{ 1, j } = ( \widetilde{ \omega } /
\omega ) p _{ 1, j } $, we get
\begin{equation*}
  I _j = \frac{1}{2} \lvert p _{ 1 , j } \rvert ^2 + \frac{1}{2}
  \omega ^2 \lvert q _{ 1, j } \rvert ^2 = \frac{ \omega ^2 }{
    \widetilde{ \omega } ^2 } \biggl( \frac{1}{2} \lvert \widetilde{ p
  }_{ 1 , j } \rvert ^2 + \frac{1}{2}
  \widetilde{ \omega } ^2 \lvert q _{ 1, j } \rvert ^2 \biggr) =
  \frac{ \omega ^2 }{ \widetilde{ \omega } ^2 } \widetilde{ I } _j , 
\end{equation*} 
and therefore $ I _j \approx 2 \omega ^2 \lvert z _{ 1, j } \rvert ^2
$.

It follows that, if \eqref{eqn:modTrigCoefficients} is consistent with
\eqref{eqn:coefficients}, then the modified trigonometric integrator
will be consistent for the corresponding energy exchange behavior.  To
compare these, let us rewrite \eqref{eqn:modTrigCoefficients} as
\begin{equation*}
\begin{aligned}
  \delta _h ^2 y _{ h, 0 } &= g _0 \bigl( y _{ h,0 } , \gamma \omega
  ^{ - 2 } g _1 ( y _{ h, 0 } , 0 ) \bigr) + \beta \frac{ \partial ^2
    g _0 }{ \partial x _1 ^2 } ( y _{ h,0 }, 0 ) ( z _{ h,1 },
  \overline{z} _{ h, 1 } ) ,\\
  2 i \omega \dot{z} _{ h, 1 } &= \alpha \frac{ \partial g _1
  }{ \partial x _1 } ( y _{ h, 0 } , 0 ) z _{ h, 1 } ,
\end{aligned}
\end{equation*}
where $ \alpha = ( \omega / \widetilde{ \omega } ) \widetilde{ \alpha
} $, $ \beta = \widetilde{ \beta } $, and $ \gamma = ( \omega /
\widetilde{ \omega } ) ^2 \widetilde{ \gamma } $, i.e.,
\begin{equation*}
  \alpha = \frac{ \omega \psi ( h \widetilde{ \omega } ) \phi (
    h \widetilde{ \omega } ) }{ \widetilde{ \omega }
    \operatorname{sinc} (h \widetilde{ \omega } ) } , \qquad
  \beta = \phi (h \widetilde{ \omega }
  ) ^2 , \qquad \gamma = \frac{ \omega ^2 \psi ( h \widetilde{
      \omega } ) \phi (
    h \widetilde{ \omega } ) }{ \widetilde{ \omega } ^2
    \operatorname{sinc} ^2 ( \frac{1}{2} h \widetilde{ \omega } ) } .
\end{equation*} 
Hence, consistency will require $ \alpha = \beta = \gamma = 1 $.  We
now arrive at the main result of this section.

\begin{theorem}
  \label{thm:IMEX}
  The IMEX method is the unique modified trigonometric integrator
  satisfying $ \alpha = \beta = \gamma = 1 $.
\end{theorem}

\begin{proof}
  Clearly $ \beta = 1 $ if and only if $ \phi = 1 $.  Substituting
  this into $ \alpha = 1 $ and solving for the filter $ \psi $, we get
  $ \psi ( h \widetilde{ \omega } ) = ( \widetilde{ \omega } / \omega
  ) \operatorname{sinc} ( h \widetilde{ \omega } )$.  Therefore,
  \begin{equation*}
    \gamma = \frac{ \omega \operatorname{sinc} ( h \widetilde{ \omega
      } ) }{ \widetilde{ \omega }
      \operatorname{sinc} ^2 ( \frac{1}{2} h \widetilde{ \omega } ) }
    = \frac{ \omega \operatorname{sinc} ( \frac{1}{2} h \widetilde{
        \omega } ) \cos ( \frac{1}{2} h \widetilde{ \omega } ) }{
      \widetilde{ \omega } \operatorname{sinc} ^2 ( \frac{1}{2} h
      \widetilde{ \omega } ) } = \frac{ \omega \cos ( \frac{1}{2} h
      \widetilde{ \omega } ) }{ \widetilde{ \omega }
      \operatorname{sinc} ( \frac{1}{2} h \widetilde{ \omega } ) } =
    \frac{ \frac{1}{2} h \omega  }{ \tan ( \frac{1}{2} h \widetilde{
        \omega } ) } .
  \end{equation*} 
  Hence, for $ \gamma = 1 $, the modified frequency must satisfy $
  \tan ( \frac{1}{2} h \widetilde{ \omega } ) = \frac{1}{2} h \omega
  $.  Finally, applying this to the prior equation for $ \psi $, we
  get
  \begin{equation*}
    \psi ( h \widetilde{ \omega } ) = \frac{ \widetilde{ \omega }
      \operatorname{sinc} ( h \widetilde{ \omega } ) }{ \omega } =
    \frac{ \sin ( h \widetilde{ \omega })  }{ h \omega } = \frac{2  \sin
      ( \frac{1}{2} h \widetilde{ \omega } ) \cos ( \frac{1}{2} h
      \widetilde{ \omega } ) }{ 2 \tan ( \frac{1}{2} h \widetilde{
        \omega } )  } = \cos ^2 ( \tfrac{1}{2} h \widetilde{ \omega }
    ) .
  \end{equation*} 
  Therefore, $ \alpha = \beta = \gamma = 1 $ holds if and only if $
  \tan ( \frac{1}{2} h \widetilde{ \omega } ) = \frac{1}{2} h \omega
  $, $ \psi (\xi) = \cos ^2 ( \frac{1}{2} \xi ) $, and $ \phi = 1 $,
  which is precisely the IMEX method.
\end{proof}

Achieving consistency thus requires solving three equations (for $
\alpha $, $\beta$, and $\gamma$) in three unknowns ($\psi$, $\phi$,
and $ \widetilde{ \omega } $).  This is impossible for standard
trigonometric integrators, since fixing $ \widetilde{ \omega } =
\omega $ results in an overdetermined system.  However, allowing $
\widetilde{ \omega } $ to be modified introduces the missing degree of
freedom necessary to satisfy all three consistency conditions.

\subsection{Long-time near-conservation of total energy and modified
  oscillatory energy}
\label{sec:nearConservation}

Away from resonances, standard trigonometric integrators nearly
conserve the total energy $ H ( q, p ) $ and stiff oscillatory energy
$ I ( q, p ) $, up to order $ \mathcal{O} (h) $ (\citet[Chapter XIII,
Theorem 7.1]{HaLuWa2006}). In fact, they note that this result can be
refined further: under the same assumptions, trigonometric integrators
nearly conserve the related quantities
\begin{equation*}
  H ( q, p ) - \rho q _1 ^T g _1 (q) , \qquad   J ( q, p ) - \rho q _1
  ^T g _1 (q),
\end{equation*} 
each up to order $ \mathcal{O} ( h ^2 ) $, where
\begin{equation*}
  \rho = \frac{ \psi ( h \omega ) }{ \operatorname{sinc} ^2 (
    \frac{1}{2} h \omega ) } - 1 ,
\end{equation*} 
and where $ J ( q, p ) = I ( q, p ) - q _1 ^T g _1 (q) $ is called the
\emph{modified oscillatory energy}. In particular, we have $ \rho = 0
$ for Gautschi-type methods with $ \psi (\xi) = \operatorname{sinc} ^2
( \frac{1}{2} \xi ) $ (e.g., Methods A and {D} in
\autoref{tab:trigMethods}), so it follows that these methods exhibit
even better long-time energy behavior, with $ H ( q, p ) $ and $ J (
q, p ) $ nearly conserved up to order $ \mathcal{O} (h ^2 ) $. (See
\citet[Chapter XIII, Exercise 8]{HaLuWa2006}.)

We now show that this improved long-time energy behavior also holds
for the IMEX method. Observe that, since IMEX corresponds to a
trigonometric integrator with frequency $ \widetilde{ \omega } $ in
the modified coordinates $ ( q, \widetilde{ p } ) $, it follows that
the total and modified oscillatory energies,
\begin{equation*}
  \widetilde{ H } ( q , \widetilde{ p } ) - \widetilde{ \rho } q _1 ^T
  g _1 (q) , \qquad \widetilde{ J } ( q, \widetilde{ p } ) -
  \widetilde{ \rho } q _1 ^T g _1 (q) ,
\end{equation*} 
are nearly conserved up to order $ \mathcal{O} ( h ^2 ) $. Here, $
\widetilde{ H } $, $ \widetilde{ J } $, and $ \widetilde{ \rho } $ are
defined just as above, with $ \widetilde{ \omega } $ in place of
$\omega$. For the IMEX method, note also that
\begin{equation*}
  \frac{ \psi ( h \widetilde{ \omega } ) }{
    \operatorname{sinc} ^2 ( \frac{1}{2} h \widetilde{ \omega } ) }  = \frac{ \cos ^2 ( \frac{1}{2} h \widetilde{ \omega } ) }{
    \operatorname{sinc} ^2 ( \frac{1}{2} h \widetilde{ \omega } ) }  = \frac{ ( \frac{1}{2} h \widetilde{ \omega } ) ^2 }{ \tan ^2 (
    \frac{1}{2} h \widetilde{ \omega } ) ^2 }  = \frac{ (\frac{1}{2} h
    \widetilde{ \omega } )^2 }{ ( \frac{1}{2} h \omega ) ^2 } = \frac{
    \widetilde{ \omega } ^2 }{ \omega ^2 }  ,
\end{equation*} 
so $ \widetilde{ \rho } = \widetilde{ \omega } ^2 / \omega ^2 - 1
$. The following theorem expresses the true energies $ H ( q, p ) $
and $ J ( q, p ) $ in terms of their modified counterparts $
\widetilde{ H } ( q, \widetilde{ p } ) $ and $ \widetilde{ J } ( q,
\widetilde{ p } ) $, thereby yielding near-conservation of both
quantities up to $ \mathcal{O} ( h ^2 ) $.

Note that the IMEX method avoids the undesirable phenomenon of
resonance instability, since $ h \widetilde{ \omega } $ is bounded
away from nonzero integer multiples of $\pi$ whenever $ h \omega $ is
bounded. (An alternative proof for the stability of this method is
given in \citet{StGr2009}.) Therefore, this convergence result can be
stated without placing non-resonance restrictions on $ h \omega $.

\begin{theorem}
  \label{thm:energy}
  For the IMEX method,
  \begin{align*}
    H ( q, p ) &= \bigl[ \widetilde{ H } ( q, \widetilde{ p } ) -
    \widetilde{ \rho } q _1 ^T g _1 (q) \bigr] - \widetilde{ \rho } J
    ( q, p ) ,\\
    J (q,p) &= \frac{ \omega ^2 }{ \widetilde{ \omega } ^2 } \bigl[
    \widetilde{ J } ( q, \widetilde{ p } ) - \widetilde{ \rho } q _1
    ^T g _1 (q) \bigr] .
  \end{align*} 
  Consequently, both $H$ and $J$ are nearly conserved up to $
  \mathcal{O} ( h ^2 ) $ as $ h \rightarrow 0 $ for any fixed $ h
  \omega $.
\end{theorem}

\begin{proof}
  The modified Hamiltonian $ \widetilde{ H } ( q, \widetilde{ p } ) $
  only differs from $ H ( q, p ) $ in replacing $ I ( q, p ) $ by $
  \widetilde{ I } ( q, \widetilde{ p } ) = \widetilde{ \omega } ^2 /
  \omega ^2 I ( q, p ) = ( \widetilde{ \rho } + 1 ) I ( q, p )
  $. Therefore,
  \begin{align*}
    H ( q, p ) &= \widetilde{ H } ( q, \widetilde{ p } ) - \widetilde{
      I } ( q, \widetilde{ p } ) + I ( q, p ) \\
    &= \widetilde{ H } ( q, \widetilde{ p } ) - \widetilde{ \rho } I (
    q, p ) \\
    &= \widetilde{ H } ( q, \widetilde{ p } ) - \widetilde{ \rho }
    \bigl[ J ( q, p ) + q _1 ^T g _1 (q) \bigr] \\
    &= \bigl[ \widetilde{ H } ( q, \widetilde{ p } ) - \widetilde{
      \rho } q _1 ^T g _1 (q) \bigr] - \widetilde{ \rho } J ( q, p ) ,
  \end{align*} 
  which proves the first equality. By a similar calculation,
  \begin{align*}
    J ( q, p ) &= \widetilde{ J } ( q, \widetilde{ p } ) - \widetilde{
      I } ( q , \widetilde{ p } ) + I ( q, p ) \\
    &= \bigl[ \widetilde{ J } ( q, \widetilde{ p } ) - \widetilde{
      \rho } q _1 ^T g _1 (q) \bigr] - \widetilde{ \rho } J ( q, p ) ,
  \end{align*} 
  which rearranges to
  \begin{equation*}
    J ( q, p ) = ( \widetilde{ \rho } + 1 ) ^{-1} \bigl[ \widetilde{ J
    } ( q, \widetilde{ p } ) - \widetilde{\rho } q _1 ^T g _1 (q)
    \bigr] = \frac{ \omega ^2 }{ \widetilde{ \omega } ^2 }  \bigl[ \widetilde{ J
    } ( q, \widetilde{ p } ) - \widetilde{\rho } q _1 ^T g _1 (q)
    \bigr] ,
  \end{equation*} 
  yielding the second equality.  Since we have already seen that $
  \widetilde{ J } ( q, \widetilde{ p } ) - \widetilde{ \rho } q _1 ^T
  g _1 (q) $ is nearly conserved up to $ \mathcal{O} ( h ^2 ) $, it
  follows that the same is true of $J(q,p)$. Finally, since $
  \widetilde{ H } ( q, \widetilde{ p } ) - \widetilde{ \rho } q _1 ^T
  g _1 (q) $ and $ J ( q, p ) $ are nearly conserved up to $
  \mathcal{O} ( h ^2 ) $, the first equality implies that so is $ H (
  q, p ) $.
\end{proof}

\section{Numerical experiments}
\label{sec:experiments}

\subsection{The Fermi--Pasta--Ulam problem} Due to its rich multiscale
coupling behavior, a variant of the Fermi--Pasta--Ulam (FPU) problem
has become a popular highly oscillatory test problem for numerical
integrators.  The original FPU problem is due to \citet{FePaUl1955},
while the version considered here is due to \citet{GaGiMaVa1992}, and
appears extensively in \citet[I.5 and XIII]{HaLuWa2006}.

Suppose we have $ 2 \ell $ unit point masses, connected together in
series by alternating weak cubic and stiff linear springs.  Denote the
displacements of the point masses by $ q _1 , \ldots , q _{ 2 \ell }
\in \mathbb{R} $, where the endpoints $ q _0 = q _{ 2 \ell + 1 } = 0 $
are fixed, and let $ p _i = \dot{q} _i $ for $ i = 1 , \ldots, 2 n $.
In these variables, the FPU system has the Hamiltonian
\begin{equation*}
  H ( q, p ) = \frac{1}{2} \sum _{ i = 1 } ^\ell ( p _{ 2 i -1 } ^2 +
  p _{ 2 i } ^2 ) + \frac{ \omega ^2 }{ 4 } \sum _{ i = 1 } ^\ell ( q
  _{ 2 i } - q _{ 2 i -1 } ) ^2 + \sum _{ i = 0 } ^\ell ( q _{ 2 i + 1
  } - q _{ 2 i } ) ^4 .
\end{equation*}
To put this into the standard form of a highly oscillatory problem, we
follow \citet[p.~22]{HaLuWa2006} in defining the coordinate
transformation
\begin{align*} 
  x _{ 0, i } &= \frac{ q _{ 2 i } + q _{ 2 i -1 } }{ \sqrt{ 2 } } ,&
  x _{ 1, i } &= \frac{ q _{ 2 i } - q _{ 2 i -1 } }{ \sqrt{ 2 }
  } ,\\
  y _{ 0,i} &= \frac{ p _{ 2 i } + p _{ 2 i -1 } }{ \sqrt{ 2 } } , & y
  _{ 1, i } &= \frac{ p _{ 2 i } - p _{ 2 i -1 } }{ \sqrt{ 2 } },
\end{align*} 
so that the Hamiltonian becomes
\begin{multline*}
  H ( x, y) = \frac{1}{2} \sum _{ i = 1 } ^\ell ( y _{ 0,i } ^2 + y _{
    1, i } ^2 ) + \frac{ \omega ^2 }{ 2 } \sum _{ i = 1 } ^\ell x _{
    1, i } ^2 \\
  + \frac{ 1 }{ 4 } \biggl[ ( x _{ 0,1 } - x _{ 1,1 } ) ^4 + \sum _{ i
    = 1 } ^{ \ell -1 } ( x _{ 0, i + 1 } - x _{ 1,i + 1 } - x _{ 0,i }
  - x _{ 1, i } ) ^4 + ( x _{0,\ell} + x _{ 1,\ell} ) ^4 \biggr] ,
\end{multline*}
which has the desired form.

Following the numerical examples in \citet{HaLuWa2006}, we consider an
instance of the FPU problem with $ \ell = 3 $, and where the initial
conditions are given by
\begin{equation*}
  x _{ 0,1 } (0) = 1 , \qquad y _{ 0, 1 } (0) = 1 , \qquad x _{ 1, 1 }
  (0) = \omega ^{-1} , \qquad y _{ 1,1 } (0) = 1 ,
\end{equation*} 
with all other initial values set to zero.  In terms of the stiff
energies $ I _j = \frac{1}{2} ( y _{ 1,j } ^2 + \omega ^2 x _{ 1,j }
^2 ) $, where $ j = 1, 2, 3 $, these conditions initialize the FPU
system with $ I _1 = 1 $ and $ I _2 = I _3 = 0 $.  As the system
evolves dynamically, the phenomenon of slow energy exchange causes
this energy to be transferred among $ I _1 $, $ I _2 $, and $ I _3 $,
on the time scale $ \mathcal{O} ( \omega ^{-1} ) $, while the total
stiff energy $ I = I _1 + I _2 + I _3 $ remains nearly constant.

\subsection{Resonance stability and oscillatory energy deviation}
\label{sec:resonance}

\begin{figure}
\subfloat[Method A]{\includegraphics{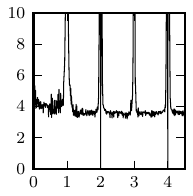}}
\subfloat[Method B]{\includegraphics{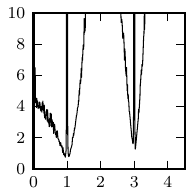}}
\subfloat[Method C]{\includegraphics{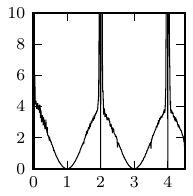}}
\subfloat[Method D]{\includegraphics{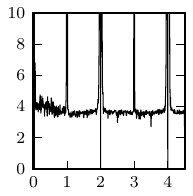}}\\
\subfloat[Method E]{\includegraphics{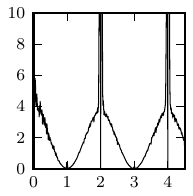}}
\subfloat[Method G]{\includegraphics{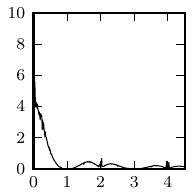}}
\subfloat[IMEX Method]{\includegraphics{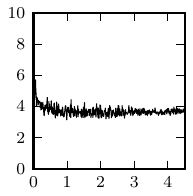}}
\subfloat[Reference]{\includegraphics{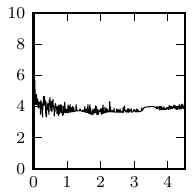}}
\caption{Maximum deviation of scaled oscillatory energy $ \omega I $
  on the time interval $ [0, 1000] $ vs.~$ h \omega / \pi $ ($ h =
  0.02 $).\label{fig:energyDeviation}}
\end{figure}

\autoref{fig:energyDeviation} depicts the maximum deviation in
frequency-scaled oscillatory energy $ \omega I $, over the time
interval $ [0,1000] $, for a range of different frequencies. As
discussed in \autoref{sec:slowExchange}, deviations in $I$ are $
\mathcal{O} ( \omega ^{-1} ) $, whereas those in $ \omega I $ are $
\mathcal{O} (1) $, making the latter more appropriate for comparison
across frequencies. (To our knowledge, the use of $ \omega I $ rather
than $I$ for numerical experiments originated in \citet{ONMc2009}.)
The time step size is fixed at $ h = 0.02 $, while $ h \omega / \pi $
ranges over $ ( 0,4.5] $.

The ``spikes'' seen at nonzero integer values of $ h \omega / \pi $
correspond to resonance instability.  Note that the energy blowup is
particularly severe for Method B (the Deuflhard/impulse method), while
Methods C and E have resonances only at even values of $ h \omega /
\pi $.  Only Method G and the IMEX method display no resonance spikes
at all.

Away from the resonance instabilities, the energy deviation behavior
is also interesting.  For all of the methods considered, $ \omega I $
appears to be $ \mathcal{O} (1) $ away from the resonance spikes.
Indeed, as long as the methods are stable (i.e., the solutions remain
bounded), we have $ I = \mathcal{O} ( \omega ^{-1} ) $, so all of the
methods nearly conserve the adiabatic invariant $I$. (This holds true
whether or not the method is consistent for the individual oscillatory
energies $ I _j $.)

However, the methods behave quite differently with respect to the
magnitude of the deviations in oscillatory energy.  For the reference
solution, the maximum deviation in $ \omega I $ is nearly constant
with respect to $\omega$, with an approximate numerical value of
$4$. Methods B, C, and E display a significant decrease in oscillatory
energy deviation near odd integer values of $ h \omega / \pi $,
indicating that the adiabatic invariant is conserved \emph{too well},
compared to the reference solution.  (This artificial
``anti-resonance'' can be seen as a sort of numerical damping.) This
behavior is even more dramatic for Method G, where the energy
deviation is artificially low for nearly all values of $ h \omega /
\pi $, not just values close to odd integers. Of the methods
considered, the IMEX method is the only one which correctly captures
the magnitude of these deviations in oscillatory energy. This
phenomenon will be revisited and analyzed in \autoref{sec:deviations},
where we will show that this behavior is governed by higher-order
terms in the modulated Fourier expansion.

\subsection{Long-time near-conservation of total energy}

\begin{figure}
\subfloat[Method A]{\includegraphics{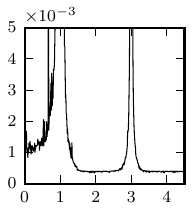}}
\subfloat[Method B]{\includegraphics{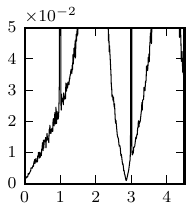}}
\subfloat[Method C]{\includegraphics{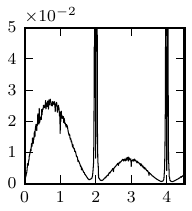}}
\subfloat[Method D]{\includegraphics{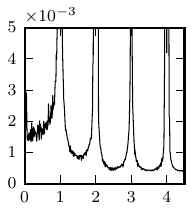}}\\
\subfloat[Method E]{\includegraphics{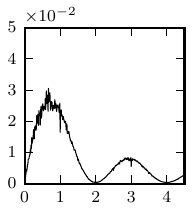}}
\subfloat[Method G]{\includegraphics{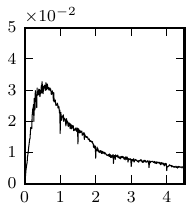}}
\subfloat[IMEX Method]{\includegraphics{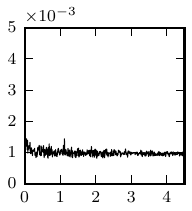}}
\phantom{\subfloat[IMEX Method]{\includegraphics{figs/totalenergyplot_method=i}}
}
\caption{Maximum deviation of total energy on the time interval $ [0,
  1000] $ vs.~$ h \omega / \pi $ ($ h = 0.02
  $). The $y$-axis is scaled 10 times smaller for Methods A,
  {D}, and IMEX due to the smaller energy deviations for these
  methods. \label{fig:totalEnergyDeviation}}
\end{figure}

\begin{figure}
\subfloat[Method A]{\includegraphics{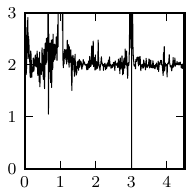}}
\subfloat[Method B]{\includegraphics{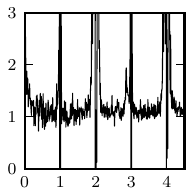}}
\subfloat[Method C]{\includegraphics{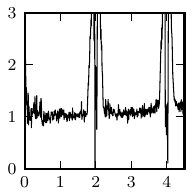}}
\subfloat[Method D]{\includegraphics{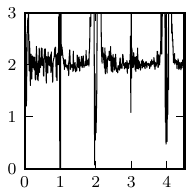}}\\
\subfloat[Method E]{\includegraphics{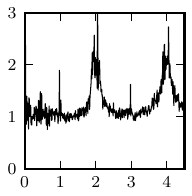}}
\subfloat[Method G]{\includegraphics{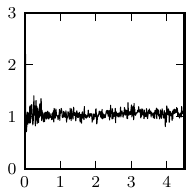}}
\subfloat[IMEX Method]{\includegraphics{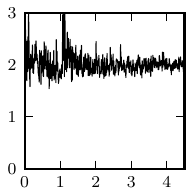}}
\phantom{\subfloat[IMEX Method]{\includegraphics{figs/totalenergyconvergenceplot_method=i}}
}
\caption{Log ratio of the maximum deviations in total energy on the
  the time interval $ [0, 1000] $, for $ h = 0.02 $ and $ h = 0.04 $,
  plotted against $ h \omega / \pi $. Away from resonances, Methods A,
  {D} and IMEX conserve total energy up to $ \mathcal{O} ( h ^2 ) $,
  while the remaining trigonometric integrators only conserve total
  energy up to $ \mathcal{O} ( h ) $. \label{fig:totalEnergyConvergence}}
\end{figure}

\autoref{fig:totalEnergyDeviation} shows the maximum deviation in
total energy (i.e., in the Hamiltonian), over the time interval $
[0,1000] $, for a range of different frequencies. (This is in contrast
to \autoref{fig:energyDeviation}, which depicted only the oscillatory
energy component of the Hamiltonian.) The reference plot is omitted,
as the exact solution preserves total energy exactly. As in
\autoref{fig:energyDeviation}, Methods A--{D} again exhibit ``spikes''
in the energy error at resonant frequencies. Notably, this is not the
case for Method E---despite the fact that it exhibited resonance
spikes for the oscillatory energy alone---nor for Method G or
IMEX. Furthermore, observe that Methods A, {D}, and IMEX conserve
energy much more closely than the other methods (at least away from
resonances), by roughly an order of magnitude. This is consistent with
the discussion in \autoref{sec:nearConservation}, including the result
in \autoref{thm:energy}, which stated that these methods conserve
total energy up to $ \mathcal{O} ( h ^2 ) $, whereas the remaining
methods only do so up to $ \mathcal{O} (h) $.

\autoref{fig:totalEnergyConvergence} illustrates the relationship of
total energy conservation to step size, plotting the log ratio of the
deviation in total energy for $ h = 0.02 $ and $ h = 0.04 $. As
anticipated by the theoretical results in
\autoref{sec:nearConservation}, including \autoref{thm:energy}, we see
that the energy deviations are $ \mathcal{O} ( h ^2 ) $ for Methods A,
{D}, and IMEX, and $ \mathcal{O} ( h ) $ for the remaining methods (at
least away from resonances).  Only the IMEX method exhibits
second-order conservation of total energy, while also remaining free
of resonance spikes.

\subsection{Slow energy exchange}

\begin{figure}
\subfloat[Method A]{\includegraphics{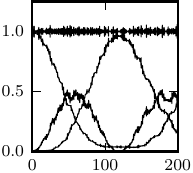}}
\subfloat[Method B]{\includegraphics{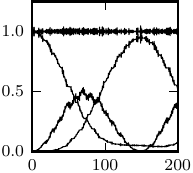}}
\subfloat[Method C]{\includegraphics{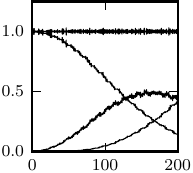}}
\subfloat[Method D]{\includegraphics{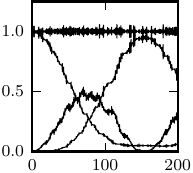}}\\
\subfloat[Method E]{\includegraphics{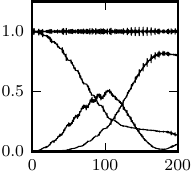}}
\subfloat[Method G]{\includegraphics{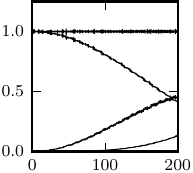}}
\subfloat[IMEX Method]{\includegraphics{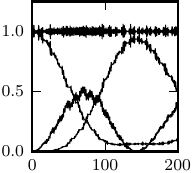}}
\subfloat[Reference]{\includegraphics{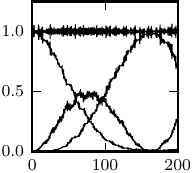}}
\caption{Slow exchange of individual and total oscillatory energies
  vs. time ($ \omega = 50 $, $ h = 0.03
  $).\label{fig:energyExchange1}}
\end{figure}

\begin{figure}
\subfloat[Method A]{\includegraphics{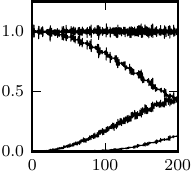}}
\subfloat[Method B]{\includegraphics{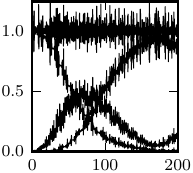}}
\subfloat[Method C]{\includegraphics{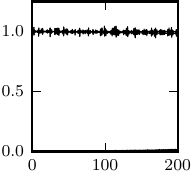}}
\subfloat[Method D]{\includegraphics{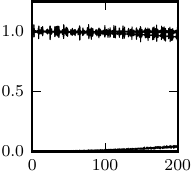}}\\
\subfloat[Method E]{\includegraphics{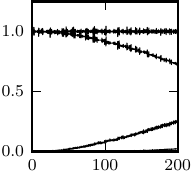}}
\subfloat[Method G]{\includegraphics{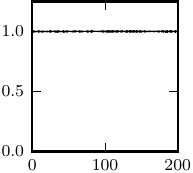}}
\subfloat[IMEX Method]{\includegraphics{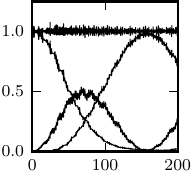}}
\subfloat[Reference]{\includegraphics{figs/energyplot_omega=50.0_reference}}
\caption{Slow exchange of individual and total oscillatory energies
  vs. time ($ \omega = 50 $, $ h = 0.1 $).\label{fig:energyExchange2}}
\end{figure}

\begin{figure}
\subfloat[Method A]{\includegraphics{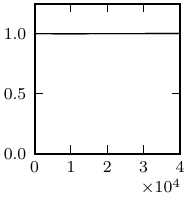}}
\subfloat[Method B]{\includegraphics{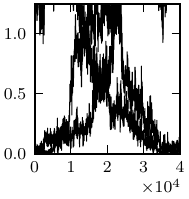}}
\subfloat[Method C]{\includegraphics{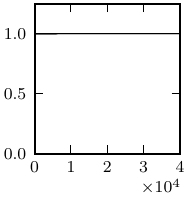}}
\subfloat[Method D]{\includegraphics{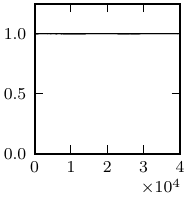}}\\
\subfloat[Method E]{\includegraphics{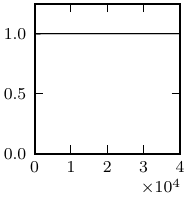}}
\subfloat[Method G]{\includegraphics{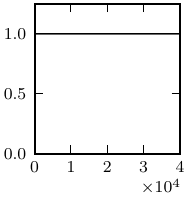}}
\subfloat[IMEX Method]{\includegraphics{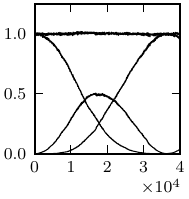}}
\subfloat[Reference]{\includegraphics{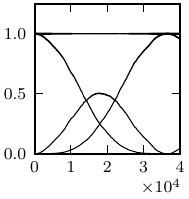}}
\caption{Slow exchange of individual and total oscillatory energies
  vs. time ($ \omega = 10000 $, $ h = 0.1 $).\label{fig:energyExchange3}}
\end{figure}

Figures~\ref{fig:energyExchange1}, \ref{fig:energyExchange2}, and
\ref{fig:energyExchange3} depict the phenomenon of slow energy
exchange for the FPU problem, following similar numerical experiments
in \citet{HaLuWa2006}. Each plot contains four curves, corresponding
to the three stiff energies, $ I _1 , I _2 , I _3 $, along with the
adiabatic invariant $ I = I _1 + I _2 + I _3 \approx 1 $.

In \autoref{fig:energyExchange1}, the parameters $ \omega = 50 $ and $
h = 0.03 $ correspond to a moderate choice of time step size: since $
h \omega / \pi \approx 0.48 $, this is prior to the onset of resonance
instability at nonzero integer values. Methods B, D, and the IMEX
method give qualitatively correct energy exchange behavior on the time
interval $ [0,200] $. By contrast, the exchange occurs too quickly for
Method A, and too slowly for Methods C, E, and G (the latter quite
dramatically).

In \autoref{fig:energyExchange2}, the fast frequency remains $ \omega
= 50 $, but we take a significantly larger time step size of $ h = 0.1
$ (with $ h \omega / \pi \approx 1.59 $). Method B and the IMEX method
still capture the correct rate of energy exchange on the time interval
$ [0,200] $, while for the other methods, the exchange occurs much too
slowly. Notice that we are also beginning to see the effects of
oscillatory energy deviation, as in
\autoref{fig:energyDeviation}. Indeed, the excessive ``noise'' visible
for Method B is due to the wide resonance band at $ h \omega / \pi = 2
$, while the pronounced lack of noise in Method G is due to its
artificially low deviations in oscillatory energy. Only the IMEX
method displays the correct oscillatory energy behavior, capturing
both the rate of exchange and the magnitude of deviations.

Next, in \autoref{fig:energyExchange3}, we depict the behavior of
these methods as they approach their high-frequency limit, keeping $ h
= 0.1 $ but taking $ \omega = 10000 $ (hence $ h \omega / \pi \approx
318 $). Since $\omega$ has been scaled by a factor of $ 200 $ compared
to the previous experiments, the time interval must also be scaled
correspondingly, so we look at energy exchange over the interval $
[0,40000] $.  As before, only Method B and the IMEX method capture the
correct rate of exchange, while for the other methods, the exchange
occurs so slowly that it cannot be seen at all on the time scale
considered.  Method B is again hampered by resonance instability, as
in \autoref{fig:energyExchange2}, which manifests as excess noise in
the oscillatory energy plot. Of the methods considered, only the IMEX
method captures the correct energy behavior in the high-frequency
limit.

\subsection{Global error in slow components}

\begin{figure}
\subfloat[Method A]{\includegraphics{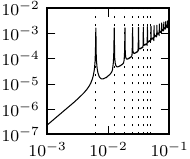}}
\subfloat[Method B]{\includegraphics{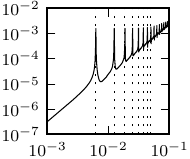}}
\subfloat[Method C]{\includegraphics{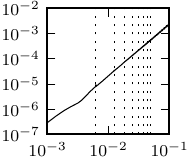}}
\subfloat[Method D]{\includegraphics{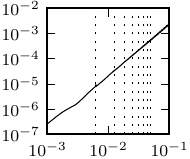}}\\
\subfloat[Method E]{\includegraphics{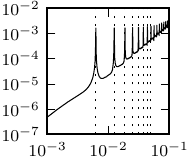}}
\subfloat[Method G]{\includegraphics{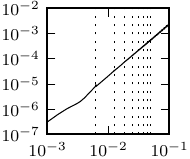}}
\subfloat[IMEX Method]{\includegraphics{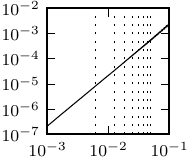}}
\subfloat[St\"ormer/Verlet]{\includegraphics{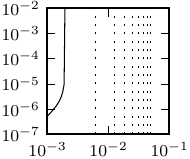}}
\caption{Global error in the slow position $ x _0 $ at the first time
  step after $ t = 1 $ vs. time step size ($ \omega = 1000
  $).\label{fig:slowPosition}}
\end{figure}

\begin{figure}
\subfloat[Method A]{\includegraphics{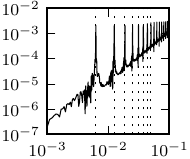}}
\subfloat[Method B]{\includegraphics{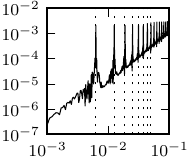}}
\subfloat[Method C]{\includegraphics{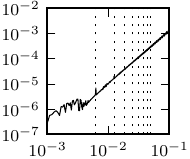}}
\subfloat[Method D]{\includegraphics{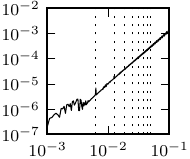}}\\
\subfloat[Method E]{\includegraphics{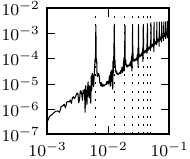}}
\subfloat[Method G]{\includegraphics{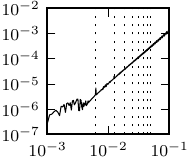}}
\subfloat[IMEX Method]{\includegraphics{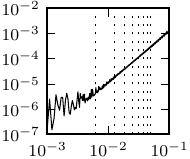}}
\subfloat[St\"ormer/Verlet]{\includegraphics{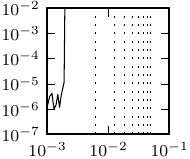}}
\caption{Global error in the slow momentum $ y _0 $ at the first time
  step after $ t = 1 $ vs. time step size ($ \omega = 1000
  $).\label{fig:slowMomentum}}
\end{figure}

Finally, in Figures~\ref{fig:slowPosition} and \ref{fig:slowMomentum},
we investigate the global error in the slow position $ x _0 $ and slow
momentum $ y _0 $, plotted against the time step size $h$. (These are
the only components of real concern: indeed, trigonometric methods are
designed precisely for problems where we are not interested in
resolving the fast oscillations.) For both figures, the error plotted
is the Euclidean distance between the numerical solution and a
reference solution, taken at the first time step after $ t = 1
$. Dotted vertical lines indicate values for which $ h \omega $ is an
integer multiple of $ 2 \pi $, where some of the methods suffer from
resonance instabilities that manifest as spikes in the global
error. Note that the St\"ormer/Verlet blows up due to linear
instability near $ h = 2/ \omega = 2 \times 10 ^{ - 3 } $,
illustrating its unsuitability for this type of highly oscillatory
problem.

Using the modulated Fourier expansion, the analysis of \citet[Chapter
XIII, Theorem 4.1]{HaLuWa2006} shows that each of the trigonometric
methods is second-order, as long as $ h \omega $ is bounded away from
an integer multiple of $\pi$. However, near these resonance points,
the order of accuracy reduces to one unless the filter functions
satisfy certain conditions. (See also \citet{GrHo2006}.)  Among the
standard trigonometric integrators, Methods C, {D}, and G satisfy
these conditions, and hence are second-order uniformly in $ h \omega
$. By contrast, Methods A, B, and E do not satisfy these conditions,
and the resulting error spikes (visible in
Figures~\ref{fig:slowPosition} and \ref{fig:slowMomentum}) lead to a
reduction in their uniform order of accuracy.

By construction, however, the IMEX method always has $ h \widetilde{
  \omega } $ bounded away from nonzero integer multiples of
$\pi$. Indeed, since $ h \widetilde{ \omega } = 2 \arctan (
\frac{1}{2} h \omega ) $, we have $ - \pi < h \widetilde{ \omega } <
\pi $, and $ h \widetilde{ \omega } $ approaches $ \pm \pi $ only in
the limit as $ h \omega $ approaches $ \pm \infty $. It follows that $
h \widetilde{ \omega } $ is bounded away from nonzero integer
multiples of $\pi$ whenever $ h \omega $ is bounded. Hence, comparing
the modulated Fourier coefficients, the argument of \citet{HaLuWa2006}
implies that the global error in the slow components for the IMEX
method is second-order, uniformly in $ h \omega $ on any bounded
region $ \lvert h \omega \rvert \leq M $.

Combined with the previous results, we remark that only Method G and
the IMEX method are uniformly second-order and free of resonance
instabilities. Of these two methods, however, only IMEX captures the
correct stiff energy behavior.

\section{Analysis of oscillatory energy deviations}
\label{sec:deviations}

We observed in \autoref{sec:resonance} that the oscillatory energy $ I
( x, \dot{x} ) $ is nearly conserved, for long times, by all the
methods considered. This is proved in \citet[Chapter XIII, Theorem
7.1]{HaLuWa2006}, where it is shown that the deviations in oscillatory
energy are $ \mathcal{O} (h) $. However, from
\autoref{fig:energyDeviation}, it is also apparent that the magnitude
of these deviations, for certain methods, is very different from the
correct value displayed in the reference solution. This can be
explained by carrying the modulated Fourier expansion to one more
term.

For the exact solution, the system of modulated Fourier coefficients
is known to have a formal invariant $ \mathcal{I} = - i \omega (
\overline{ u } ^T \dot{ u } - u ^T \dot{ \overline{ u } } ) +
\mathcal{O} ( \omega ^{ - 6 } ) $, where $ u = e ^{ i \omega t } z
$. A higher-order expansion of this invariant appears in
\citep[Chapter XIII, Equation 6.12]{HaLuWa2006}, and the $ \mathcal{O}
( \omega ^{ - 6 } ) $ estimate is obtained by observing from
\citep[Chapter XIII, Equation 5.3]{HaLuWa2006} that the remainder is a
product of $- i \omega$ with two additional factors, which are
respectively $ \mathcal{O} ( \omega ^{ - 4 } ) $ and $ \mathcal{O} (
\omega ^{ - 3 } ) $.  Now, with $ \dot{u} = e ^{ i \omega t } (
\dot{z} + i \omega z ) = i \omega e ^{ i \omega t } z + \mathcal{O} (
\omega ^{ - 2 } ) $, this gives
\begin{equation*}
  \mathcal{I} = 2 \omega ^2 \lVert z _1 \rVert ^2 + \mathcal{O} (
  \omega ^{ - 2 } ) .
\end{equation*} 
This should be compared to the oscillatory energy, which is
\begin{equation*}
  I ( x , \dot{x}) = \frac{1}{2} \lVert \dot{ x } _1 \rVert ^2 +
  \frac{1}{2} \omega ^2 \lVert x _1 \rVert ^2 .
\end{equation*} 
To do this, consider the expansions
\begin{align*}
  x _1 &= y _1 + e ^{ i \omega t } z _1 + e ^{ - i \omega t }
  \overline{ z } _1 + \mathcal{O} ( \omega ^{ - 4 } ) \\
  \dot{x} _1 &= \dot{y} _1 + e ^{ i \omega t } ( z _1 + i \omega z _1
  ) + e ^{ - i \omega t } ( \overline{ z } _1 - i \omega \overline{ z
  } _1 ) + \mathcal{O} ( \omega ^{ - 4 } ) .
\end{align*} 
Inserting  $ y _1 = \mathcal{O} ( \omega ^{ - 2 } ) $, $
z _1 = \mathcal{O} ( \omega ^{ - 1 } ) $, $ \dot{z} _1 = \mathcal{O} (
\omega ^{ - 2 } ) $, we get the estimate
\begin{align*}
  I ( x, \dot{x}) &= \frac{1}{2} \omega ^2 \lVert z _1 - \dot{z} _1
  \rVert ^2 + \frac{1}{2} \omega ^2 \lVert y _1 + e ^{ i \omega t } z
  _1 + e ^{ - i \omega t } \overline{ z } _1 \rVert ^2 + \mathcal{O} (
  \omega ^{ - 2 } ) \\
  &= 2 \omega ^2 \bigl( \lVert z _1 \rVert ^2 + y _1 ^T \operatorname{
    Re } ( e ^{ i \omega t } z _1 ) \bigr) + \mathcal{O} ( \omega ^{ -
    4 } ) \\
  &= \mathcal{I} + 2 \omega ^2 y _1 ^T \operatorname{ Re } ( e ^{ i
    \omega t } z _1 ) + \mathcal{O} ( \omega ^{ - 2 } ) .
\end{align*} 
Since $\mathcal{I}$ is (nearly) conserved over long times, the second
term controls the $ \mathcal{O} ( \omega ^{-1} ) $ deviations in the
oscillatory energy $I$. This term contains two fluctuating components,
corresponding to the evolution of $ y _1 = \omega ^{ - 2 } g _1 (y) $
and $ 2 i \omega \dot{z} _1 = g _1 ^\prime (y) z = g _{ 1, 1 } ( y _0,
0 ) z _1 + \mathcal{O} ( \omega ^{ - 2 } ) $.  Since the latter is
controlled by the formal invariant $\mathcal{I}$, it follows that
\begin{equation*}
  \lvert I - \mathcal{I} \rvert \leq \frac{ \sqrt{ \mathcal{I} }}{
      \sqrt{ 2 } } \omega ^{-1} \bigl\lVert g _{ 1, 1 } ( y _0, 0 )
    \bigr\rVert + \mathcal{O} ( \omega ^{ - 2 } ) .
  \end{equation*} 

Repeating the above estimates for modified trigonometric integrators
gives the formal invariant
\begin{equation*}
  \mathcal{I} _h = 2 \omega ^2 \lVert z _{ h , 1 } \rVert +
  \mathcal{O} ( \omega ^{ - 2 } ) ,
\end{equation*} 
for $ h \rightarrow 0 $, $ h \omega $ fixed, which is related to $I$
by
\begin{equation*}
  I ( x _n , \dot{x} _n ) = \mathcal{I} _h + 2 \omega ^2 y _{ h, 1 }
  ^T \operatorname{ Re } ( e ^{ i \omega t } z _{ h , 1 } ) +
  \mathcal{O} ( \omega ^{ - 2 } ) .
\end{equation*} 
From the modulated Fourier expansion, we have
\begin{equation*}
  y _{ h , 1 } = \gamma \phi ( h \omega ) ^{-1} y _1 + \mathcal{O} (
  \omega ^{ - 3 } ) , \qquad \dot{z} _{ h , 1 } = \alpha \dot{z} _1 +
  \mathcal{O} ( \omega ^{ - 2 } ) .
\end{equation*} 
Thus, although the $ z _{ h, 1 } $ factor may have an incorrect
evolution on the $ \mathcal{O} ( \omega ^{-1} ) $ time scale if $
\alpha \neq 1 $, in the neighborhood of a particular solution $ (y,z)
$ on the $ \mathcal{O} (1) $ time scale, it does not affect the
deviations in $I$. Rather, these are controlled by the first factor,
and are therefore correct to leading order if and only if $ \gamma =
\phi $, i.e., $ \omega ^2 \psi ( h \widetilde{ \omega } ) =
\widetilde{ \omega } ^2 \operatorname{sinc} ^2 ( \frac{1}{2} h
\widetilde{ \omega } ) $.

For standard trigonometric integrators, this is true for the
Gautschi-type methods with $ \psi (\xi) = \operatorname{sinc} ^2
(\frac{1}{2} h \omega ) $, i.e., for Method A (with $ \phi = 1 $) and
Method {D} (with $ \phi \neq 1 $). The IMEX method also satisfies this
consistency condition, since $ \gamma = \phi = 1 $. For the remaining
methods, the observed deviations in \autoref{fig:energyDeviation} are
correct up to the factor $ \gamma / \phi $ calculated above. For
example, Methods C and E have $ \psi (\xi) = \operatorname{sinc} ^2
(\xi) $, so
\begin{equation*}
  \frac{ \gamma (\xi) }{ \phi (\xi) } = \frac{ \operatorname{sinc}
    ^2 (\xi) }{ \operatorname{sinc} ^2  (\tfrac{1}{2} \xi ) } = \frac{
    \operatorname{sinc} ^2 ( \tfrac{1}{2} \xi ) \cos ^2 ( \tfrac{1}{2}
    \xi ) }{ \operatorname{sinc} ^2 ( \tfrac{1}{2} \xi ) } = \cos ^2
  (  \tfrac{1}{2} \xi  ) ,
\end{equation*} 
which is clearly visible, in \autoref{fig:energyDeviation}, as
period-$2 \pi $ oscillations in the magnitude of energy deviation.  On
the other hand, for Method G, we have the filter $ \psi (\xi) =
\operatorname{sinc} ^3 (\xi) $, so
\begin{equation*}
  \frac{ \gamma (\xi) }{ \phi (\xi) } = \frac{ \operatorname{sinc} ^3
    (\xi) }{ \operatorname{sinc} ^2 ( \tfrac{1}{2} \xi ) } = \frac{
    \operatorname{sinc} ^3 ( \tfrac{1}{2} \xi ) \cos ^3 ( \tfrac{1}{2}
    \xi ) }{ \operatorname{sinc} ^2 ( \tfrac{1}{2} \xi ) } =
  \operatorname{sinc} ( \tfrac{1}{2} \xi ) \cos ^3 ( \tfrac{1}{2} \xi
  ) ,
\end{equation*} 
which leads to rapid decay in the magnitude of energy deviation. 

\section{Conclusion}

This paper was motivated by the fact that, while conventional
trigonometric integrators have many desirable properties---especially,
with respect to stability, accuracy, and energy behavior---there are
``no-go theorems'' making it impossible for any single integrator to
have these good properties simultaneously. Other work, particularly on
multi-force methods, showed a way around these obstacles, but at the
cost of several nonlinear force evaluations per time step. On the
other hand, the observations of \citet{StGr2009} regarding the IMEX
method suggested that, by modifying the fast frequency, one might find
another way around these obstacles, without suffering the greater
computational cost required by multi-force methods.

By extending the modulated Fourier expansion techniques of
\citet{HaLu2000,HaLuWa2006}, we have shown that, for \emph{modified
  trigonometric integrators}, it is indeed possible to get around
these no-go theorems---and that the IMEX method is in fact the unique
modified trigonometric integrator which correctly models the
multiscale phenomenon of slow energy exchange.  Moreover, the IMEX
method maintains desirable properties with respect to resonance
stability and preservation of adiabatic invariants, while also being
uniformly of second-order accuracy in global error, and does so
without any additional computational cost relative to conventional
trigonometric integrators. Finally, we have shown that while all of
these integrators exhibit near-conservation of oscillatory energy,
only some of them---in particular, the Gautschi-type trigonometric
integrators and the IMEX method---consistently model the magnitude of
deviations in this adiabatic invariant.

\section*{Acknowledgments}

The authors wish to thank the organizers and participants in the 2011
Oberwolfach Workshop on Geometric Numerical Integration---especially
Ernst Hairer, Marlis Hochbruck, and Christian Lubich---for their
valuable feedback on this work. We also thank the anonymous referees
for their helpful comments and suggestions. R.~M.~was supported by a
grant from the Marsden Fund of the Royal Society of New
Zealand. A.~S.~gratefully acknowledges partial travel support from
Reinout Quispel, along with an AMS--Simons travel grant and an
Oberwolfach Leibniz grant.


\end{document}